\documentclass[a4,12pt,reqno]{amsart}
\usepackage{amsmath}
\usepackage[english]{babel} 
\usepackage{amsthm}
\usepackage{mathalfa}
\usepackage{frontespizio}
\usepackage{amsfonts}
\usepackage{amssymb}
\usepackage{lmodern}
\usepackage[T1]{fontenc}
\usepackage[utf8]{inputenc} 
\usepackage{latexsym} 
\usepackage{graphics} 
\usepackage{euscript} 
\usepackage{mathrsfs} 
\usepackage{verbatim}
\usepackage[dvips]{curves}
\usepackage{tikz-cd}
\usepackage{graphicx}
\usepackage{subfigure}
\usepackage{color}
\usepackage{mathtools}
\usepackage[a4paper,top=2.5cm,bottom=2.5cm,left=2cm,right=2cm]{geometry}

\newcommand{ \Rn} {\mathbb{R}^n}
\newcommand{ \Rm} {\mathbb{R}^m}
\newcommand{ \Ru} {\mathbb{R}}

\newcommand{\Om}{\Omega}

\newcommand{\mA}{\mathcal{A}}
\newcommand{\mB}{\mathcal{B}}

\newcommand{\scu}{\longrightarrow}
\newcommand{\so}[1]{W^{1,p}(#1)}

\newcommand{\lp}[1]{L^{p}(#1)}

\newcommand{\sol}[1]{W^{1,p}_{loc}(#1)}

\newcommand{\anso}[1]{W^{1,p}_X(#1)}

\newcommand{\ciuno}[1]{C^{1}(\overline{#1})}
\newcommand{\lul}{L^1_{loc}(\Om)}
\newcommand{\lu}{L^1(\Om)}
\newcommand{\eps}{\varepsilon}

\newcommand{\cla}{I_{m,p}(a,c_0,c_1,c_2)}
\newcommand{\clam}{\mathcal{M}_{p}(d_1,d_2,d_3,d_4,\mu)}
\newcommand{\claf}{\mathcal{K}_{m,p}(a,c_0,c_1,c_2)}
\newcommand{\clad}{J_{m,p}(a,c_0,c_1,c_2)}
\newcommand{\clat}{K_{m,p}(a,c_0,c_1,c_2)}
\newcommand{\clafd}{\mathcal{J}_{m,p}(a,c_0,c_1,c_2,q,\omega)}
\newcommand{\claft}{\mathcal{I}_{m,p}(a,c_0,c_1,c_2)}
\newcommand{\clagt}{\mathcal{U}_{m,p}(a,c_1,c_2)}
\newcommand{\clafto}{\mathcal{W}_{m,p}(a,c_1,c_2,\omega)}
\newcommand{\clag}{\mathcal{V}_{m,p}(a,c_1,c_2)}
\newcommand{\claftot}{\mathcal{W}_{m,p}(a,c_0,c_1,c_2,\Tilde{\omega})}

\newcommand{\winf}[1]{W^{1,\infty}(#1)}

\newtheorem{thm}{Theorem}[section]
\newtheorem{prop}[thm]{Proposition}

\newtheorem{deff}[thm]{Definition}

\theoremstyle{definition}
\newtheorem*{rem}{Remark}

\newcommand{\average}{{\mathchoice {\kern1ex\vcenter{\hrule height.4pt
				width 6pt
				depth0pt} \kern-9.7pt} {\kern1ex\vcenter{\hrule height.4pt width 4.3pt
				depth0pt}
			\kern-7pt} {} {} }}

\newcommand{\normaansoom}[1]{\Vert#1\Vert_{\anso{\Om}}}

\DeclareMathOperator{\diver}{div} 
\DeclareMathOperator{\spann}{span} 
\DeclareMathOperator{\lie}{Lie} 
\title[$\Gamma$-Compactness of Some Classes of Integral Functionals]{$\Gamma$-Compactness of Some Classes of Integral Functionals Depending on Vector Fields}
\author[Fares Essebei]{Fares Essebei}
\address[Fares Essebei]{Department of Mathematics, University of Trento, Via Sommarive 14, 38123 Povo (Trento), Italy}
\email[Fares Essebei]{fares.essebei@unitn.it}
\author[Simone Verzellesi]{Simone Verzellesi}
\address[Simone Verzellesi]{Department of Mathematics, University of Trento, Via Sommarive 14, 38123 Povo (Trento), Italy}
\email[Simone Verzellesi]{simone.verzellesi@unitn.it}
\numberwithin{equation}{section}

\begin{document}
\maketitle
\begin{abstract}
    In this paper we achieve some $\Gamma$-compactness results for suitable classes of integral functionals depending on a given family of Lipschitz vector fields, with respect to both the strong $L^p-$topology and the strong $W_X^{1,p}-$topology.
\end{abstract}
\section{Introduction}
Starting form the seminal works by E. De Giorgi and T. Franzoni (\cite{degiorgisolo,degiorgi}), the study of $\Gamma$-convergence has pervaded the evolution of modern mathematical analysis, and has developed in several different directions, exhibiting important applications to many branches of calculus of variations, such as homogenization, minimal surfaces and partial differential equations.
For an exhaustive introduction to this topic, we refer to the monographs \cite{dalmaso, braides, braidefra}. \\
Since the late 1970s, G. Buttazzo and G. Dal Maso has investigated $\Gamma$-convergence in the framework of Lebesgue spaces, Sobolev spaces and $BV$ spaces (cf. for instance \cite{buttazzodalmaso1, buttazzodalmaso3, dalmaso2}). More recently, in \cite{franchiserapserra}, the authors started the investigation of variational functionals depending on suitable families of Lipschitz vector fields. 
By a family of Lipschitz vector fields we mean an $m-$tuple $X=(X_1,\ldots, X_m)$, with $m\leq n$, where each $X_j$ is a first-order differential operator with Lipschitz coefficients $c_{j,i}$
defined on a bounded open set $\Omega\subseteq\Rn$, i.e.
\[
X_j(x)=\,\sum_{i=1}^nc_{j,i}(x)\partial_i\quad j=1,\dots,m.
\] 
Since that and other important works (cf. for instance \cite{garofalo}), the possibility to extend classical results to this new framework has been extensively studied in many papers. For instance, many homogenization problems have been solved in the special setting of the Heisenberg group (cf. \cite{ddmm}) and in more general Carnot Groups (cf. \cite{BMT, MV, FT}). In the last years, in \cite{maionepinafsk,maionepinafsk2} the authors started the investigation of the $\Gamma-$convergence of translations-invariant local functionals $F:\,L^p(\Omega)\times\mA\to [0,\infty]$, 
where $1<p<+\infty$ and $\mA$ is the class of all open subsets of $\Om$, defined starting from  a family of Lipschitz vector fields satisfying the so-called \emph{linear independence condition (LIC)}, which will be recalled in the following section. Their strategy strongly relies on the possibility to represent, under suitable conditions, an abstract translations-invariant local functional as an integral functional. This tool, which was obtained as a generalization of the Euclidean result contained in \cite{buttazzodalmaso3}, allowed the authors to achieve a $\Gamma$-compactness property for suitable classes of translations-invariant integral functionals. Afterwards, inspired by \cite{maionepinafsk}, in \cite{essverzpin} the authors took into account the problem of representing in integral form a local functional for which the translations-invariant assumption is dropped, extending the integral representation results presented in \cite{buttazzodalmaso1, buttazzodalmaso2} to the non-Euclidean framework of Lipschitz vector fields. According to \cite{buttazzodalmaso1, buttazzodalmaso2}, the authors of \cite{essverzpin} took into account both the convex and the non-convex case. For dealing with the latter situation, they exploited some continuity condition introduced in \cite{buttazzodalmaso2} known as \emph{weak condition $(\omega)$} and \emph{strong condition $(\omega)$}.\\

The aim of the present paper is to generalize the $\Gamma$-compactness result presented in \cite{maionepinafsk} to different classes of integral functionals which are not assumed to be translations-invariant, and which can be defined both on $\lp{\Om}$ and on $\anso{\Om}$. To be more precise, we are going to show the three following results.
\begin{itemize}
    \item $\Gamma(L^p)$-compactness, under standard boundedness and coercivity requirements, for a class of non-negative convex integral functionals defined on $\lp{\Om}\times\mA$.
     \item $\Gamma(W_X^{1,p})$-compactness, under standard boundedness requirements, for a class of non-negative convex integral functionals defined on $\anso{\Om}\times\mA$.
     \item $\Gamma(W_X^{1,p})$-compactness, under standard boundedness requirements, for a class of non-negative and possibly non-convex integral functionals defined on $\anso{\Om}\times\mA$ which satisfies the \emph{strong condition $(\omega X)$} uniformly on the class.
\end{itemize}
Here the strong condition $(\omega X)$ is a suitable continuity condition which is strongly inspired by the classical notion introduced in \cite{buttazzodalmaso2}, and that will be properly defined in Section 4.
As one could expect, the lack of translations-invariance implies in general a dependence of the Lagrangian on the function variable.
Moreover we point out that in the last two cases no coercivity assumption is requested, and differently from the $L^p$ situation, we can allow also the case in which $p=1$.\\

Our general strategy is classical and consists of two main steps:
\begin{itemize}
    \item[\textbf{Step 1}] given a sequence $(F_h)_h$ in an appropriate class of integral functional $\mathcal{I}$, find a subsequence $(F_{h_k})_k$ and a local functional $F$ such that 
    \begin{equation*}
        F(\cdot,A)=\Gamma-\lim_{k\to\infty}F_{h_k}(\cdot,A)
    \end{equation*}
    for any $A\in\mA$, and moreover show that such an $F$ satisfies some structural properties;\\
        \item[\textbf{Step 2}] choose a suitable subclass $\mathcal{J}\subseteq\mathcal{I}$ and show that, whenever $(F_h)_h$ belongs to $\mathcal{J}$, then F belongs to $\mathcal{J}$.\\
\end{itemize}
Working in the $L^p$ framework the approach is quite standard. Indeed, for achieving Step 1 we exploits classical results of $\Gamma$-convergence in $L^p$, for which we refer to \cite{dalmaso}, and some properties of the $X$-gradient, whose definition will be given soon after this introduction. On the other hand, Step 2 is based on the new integral representation result for convex local functionals introduced in \cite[Theorem 2.3]{essverzpin}, and consists in verifying that the abstract $\Gamma$-limit $F$ satisfies the hypotheses relative to $\mathcal{J}$. \\

When instead we consider functionals defined on $\anso{\Om}$ and we perform the $\Gamma$-limit with respect to the strong topology of $\anso{\Om}$, the situation is more delicate. In particular, in order to achieve Step 1, we need to understand how to modify some arguments of \cite{dalmaso}. More precisely, we introduce a suitable notion of \emph{uniform fundamental estimate} which is inspired by the classical notion of fundamental estimate but which turns out to be more useful for our purposes. Indeed it allows us to drop the coercivity assumptions, and to mimicking the results that allowed the conclusion of Step 1 in the $L^p$ case, adapting them to this new framework.
Again, Step 2 relies on the possibility to exploit \cite[Theorem 2.3]{essverzpin} and the integral representation result for non-convex integral functionals proved in \cite[Theorem 4.4]{essverzpin} to represent the $\Gamma(W_X^{1,p})-$limit
in integral form. To this aim we show that the strong condition $(\omega X)$, which implies one of the requirement of the latter result, well behaves with respect to the passage to the $\Gamma$-limit, provided we perform this operation with respect to the strong topology of $\anso{\Om}$. \\
For the sake of completeness, we want to point out that some of the results achieved in the non-Euclidean framework were, to our knowledge, unsolved even in the classical Sobolev setting $W^{1,p}(\Omega)$.

To conclude, in the final section we list some remarks and problems that are still opens. In particular, we show some critical aspects and we prove some results with the hope that they may be useful to anyone who will try to handle this analysis.\\

This paper is organized as follows. In Section 2 we collect some preliminaries about Lipschitz vector fields and related functional spaces, local functionals and $\Gamma$-convergence. In Section 3 we prove the $\Gamma$-compactness result in the $L^p$ framework. In Section 4 we show two $\Gamma-$compactness results in the $W_X^{1,p}$ setting. In Section 5 we make some further remarks and we explain which problems are still open.\\

\textit{Acknowledgements}. We would like to thank M. Bonacini, G. Dal Maso, A. Defranceschi and A. Pinamonti for useful
suggestions and discussions on these topics.

\section{Preliminaries} 
\subsection{Notation}
We adopt the convention $\infty:=+\infty$. Unless otherwise specified, we let $1\leq p<\infty$ and $m,n\in\mathbb{N}\setminus\{0\}$ with $m\leq n$, we denote by $\Om$ an open and bounded subset of $\Rn$, by $\mA$ the family of all open subsets of $\Om$ and by $\mathcal{B}$ the family of all Borel subsets of $\Om$. Given two sets $A$ and $B$, we write $A\Subset B$ whenever $\overline{A}\subseteq B$. We set $\mA_0$ to be the subfamily of $\mA$ of all the open subsets $A$ of $\Om$ such that $A\Subset\Om$ and by $\mathcal{B}_0$ the subfamily of $\mathcal{B}$ of all the Borel subsets $B$ of $\Om$ such that $B\Subset \Om$. For any $v\in\Rn$, we denote by $|v|$ the Euclidean norm of $v$. We denote by $\mathcal{L}^n$ the restriction to $\Om$ of the $n$-th dimensional Lebesgue measure, and for any set $E\subseteq\Om$ we write $|E|:=\mathcal{L}^n(E)$.
Throughout this paper, we mean gradients as row vectors.
\subsection{Vector Fields}
In what follows we identify
a first-order differential operator  $\sum_{i=1}^n c_{i}\frac{\partial}{\partial x_i}$ with the map $(c_1(x),\ldots,c_n(x)):\Om\to \mathbb{R}^n$.
Given a family $X:=(X_1,\ldots,X_m)$ of Lipschitz vector fields on $\Omega$, we denote by $C(x)$ the $m\times n$ matrix defined as \[C(x):=[c_{j,i}(x)]_{\substack{{i=1,\dots,n}\\{j=1,\dots,
m}}}\]
Throughout this work we assume that $X$ satisfies the so-called \emph{linear independence condition (LIC)} on $\Om$, i.e. we assume that the set $$N_X:=\{x\in\Om\,:\,X_1(x),\ldots,X_m(x)\text{ are linearly dependent}\}$$ has Lebesgue measure zero. 
We point out that this condition is quite general and embraces many relevant families of vector fields studied in literature. In particular neither the {\it H\"ormander condition}, that is, each vector field $X_j$ is smooth and it holds that
$$\spann\{\lie(X_1(x),\ldots,X_m(x))\}=\Rn\qquad \text{ for any }x\in \Om,$$
nor the weaker assumption that $X$ induces a Carnot--Carath\'eodory metric in $\Omega$, is requested. For further accounts on this topic we refer to \cite{BLU}.
If $u\in\lul$, we define its distributional \emph{$X$-gradient} as 
\begin{equation*}
    \langle Xu,\varphi\rangle:=-\int_U u \diver (\varphi(x)\cdot C(x)){\rm d} x\qquad\text{ for any }\varphi\in C^\infty_c(\Om,\Rm).
\end{equation*}
In the particular case in which $X$ is the family of \emph{horizontal vector fields} of a Carnot Group, then the $X$-gradient reduces to the classical \emph{horizontal gradient} (cf. \cite{BLU}).
Moreover, we define the vector space
	$$\anso{\Om}:=\{u\in L^p(\Om)\,:\,Xu\in L^p(\Om)\},$$
and we refer to it as \emph{$X$-Sobolev space}, and to its elements as \emph{$X$-Sobolev functions}.
It is well known (cf. \cite{follandstein}) that the vector space $\anso{\Om}$, endowed with the norm
	$$\Vert u\Vert_{\anso{\Om}}:=\Vert u\Vert_{L^p(\Om)}+\Vert Xu\Vert_{L^p(\Om,\Rm)},$$
is a Banach space for any $p\in[1,+\infty]$, and that it is reflexive whenever $1<p<+\infty$.
The following simple proposition tells us that $X$-Sobolev spaces are actually a generalization of the classical ones, and that the $X$-gradient can be computed starting from the Euclidean gradient in a very simple way, whenever the function is regular enough.

\begin{prop}\label{503soinanso}
	The following facts hold:
	\begin{itemize}
		\item [(i)] if $n=m$ and $c_{j,i}(x)=\delta_{j,i}$ for every $i,j=1,\ldots, n$, then $\so{\Om}=\anso{\Om}$;
		\item[(ii)] $\so{\Om}\subseteq\anso{\Om}$, the inclusion is continuous and 
		\[Xu(x)=Du(x)\cdot C(x)^T\]
		for every $u\in\so{\Om}$ and a.e. $x\in\Omega$.
	\end{itemize}
\end{prop}
Let us notice that, being $\Om$ bounded, for any family $X$ of Lipschitz vector fields we have that
$$\winf{\Om}\subseteq\so{\Om}\subseteq\anso{\Om}.$$
Finally, similarly to the Euclidean case, a Meyers--Serrin approximation result holds even in this non-Euclidean framework
(cf. \cite[Theorem 1.2]{FSSC2}).
\begin{thm}\label{2402meyersserrinx}
	Let $\Om$ be an open subset of $\Rn$, and
	let $1\leq p<+\infty$. Then
	$$\overline{\anso{\Om}\cap C^\infty(\Om)}=\anso{\Om},$$ 
	where the closure is w.r.t. the metric topology of $(\anso{\Om},\normaansoom{\cdot})$.
\end{thm}
We conclude this section with a Leibniz-type property for the $X$-gradient, which is a direct consequence of the previous result and which will be very useful in the sequel. 
\begin{prop}\label{708leibniz}
For any $u,v\in\anso{\Om}$, it holds that
\begin{equation*}
    X(uv)=(Xu)v+u(Xv)\qquad\text{a.e. on $\Om$.}
\end{equation*}
\end{prop}
\begin{proof}
Assume first that $u,v\in\anso{\Om}\cap C^\infty(\Om)$. Then it follows that
\begin{equation}\label{708leibniz1}
\begin{split}
    X(uv)&= D(uv)\cdot C^T=[(Du)v+u(Dv)]\cdot C^T\\
    &=Du\cdot C^Tv+uDv\cdot C^T=(Xu)v+u(Xv)
\end{split}
\end{equation}
everywhere on $\Om$.
Let now $A'\Subset\Om$, $u\in\anso{\Om}$ and $v\in\anso{\Om}\cap C^\infty(\Om)$. From Theorem \ref{2402meyersserrinx} we know in particular that there exists a sequence $(u_h)_h\subseteq\anso{\Om}\cap C^\infty(\Om)$ converging to $u$ in the strong topology of $\anso{A'}$, and clearly $v\in C^\infty(\overline{A'})$. It is easy to see that the sequence $(v u_h)_h$ belongs to $\anso{A'}\cap C^\infty(\overline{A'})$ and converges to $uv$ in the strong topology of $\anso{A'}$. This fact, together with \eqref{708leibniz1} and recalling that $\sup_{A'}|Xv|<+\infty$ since $\sup_{A'}|Dv|<+\infty$, yields that
\begin{equation*}
    \begin{split}
&\|X(uv)-(Xu)v-u(Xv)\|_{L^p(A',\Rm)}\\
&\leq\|X(uv)-X(u_hv) \|_{L^p(A',\Rm)}+\|(Xv)u_h+vX(u_h)-(Xu)v-u(Xv)\|_{L^p(A',\Rm)},\\
    \end{split}
\end{equation*}
and so, passing to the limit as $h\to\infty$, we conclude that
\begin{equation*}
    X(uv)=(Xu)v+u(Xv)\qquad\text{a.e. on $A'$.}
\end{equation*}
Since $\Om$ can be approximated by a countable family of open sets $A'\Subset\Om$, we conclude that
\begin{equation*}
    X(uv)=(Xu)v+u(Xv)\qquad\text{a.e. on $\Om$}
\end{equation*}
for any $u\in\anso{\Om}$ and $v\in\anso{\Om}\cap C^\infty(\Om)$.
Repeating once more the same procedure, the thesis follows.
\end{proof}

\subsection{Local Functionals} Before entering the core of this work, we fix some notation about local functionals, in order to ease the reading.
If we have a functional $F:\lp{\Om}\times\mA\scu[0,\infty]$ (resp. $F:\anso{\Om}\times\mA\scu[0,\infty]$), we say that $F$ is:
	\begin{itemize}
	\item [$(i)$] a \emph{measure} if $F(u,\cdot)$ is a measure for any $u\in\lp{\Om}$ (resp. $u\in\anso{\Om}$);
	\item [$(ii)$] \emph{local} if, for any $A\in\mA$ and $u,v\in\lp{\Om}$ (resp. $u,v\in\anso{\Om}$), then
	$$u|_{A}=v|_{A}\implies F(u,A)=F(v,A);$$
\item [$(iii)$] \emph{convex} on $\anso{\Om}$ if $F(\cdot,A)$ restricted to $\anso{\Om}$ is convex for any $A\in\mA$;
	\item [$(v)$] \emph{$L^p$-lower semicontinuous} (resp. \emph{$W_X^{1,p}$-lower semicontinuous}) if $F(\cdot,A)$ is  $L^p$-lower semicontinuous (resp. $W_X^{1,p}$-lower semicontinuous) for any $A\in\mA$;
	\item [$(vi)$]  \emph{weakly*- sequentially lower semicontinuous} if if $F(\cdot,A)$ restricted to ${\winf{\Om}}$ is seq. l.s.c. w.r.t. the weak*- topology of $\winf{\Om}$ for any $A\in\mA$.
	\end{itemize}

\subsection{Basic Notions of $\Gamma$-convergence}
 In this section we collect some basic definition and results about $\Gamma$-convergence, for which we refer to \cite{dalmaso}.
 We recall that, if $(X,\tau)$ is a first-countable topological space and $(F_h)_h$ is a sequence of functions defined on $(X,\tau)$ with values in $\overline{\Ru}$, we define the \emph{$\Gamma$-upper limit} and \emph{$\Gamma$-lower limit} respectively as
 \begin{equation*}
     \Gamma-\liminf_{h\to\infty}F_h(u):=\inf\left\{\liminf_{h\to\infty}F_h(u_h)\,:\,u_h\to u\right\}
 \end{equation*}
 and
 \begin{equation*}
     \Gamma-\limsup_{h\to\infty}F_h(u):=\inf\left\{\limsup_{h\to\infty}F_h(u_h)\,:\,u_h\to u\right\},
 \end{equation*}
 and we say that $(F_h)_h$ \emph{$\Gamma$-converges} to $F:(X,\tau)\scu\overline{\Ru}$ if it holds that
 \begin{equation*}
     \Gamma-\liminf_{h\to\infty}F_h(u)=\Gamma-\limsup_{h\to\infty}F_h(u)
 \end{equation*}
 for any $u\in X$. In this case we say that $F$ is the $\Gamma-$limit of $(F_h)_h$ and we write $F=\Gamma-\lim_{h\to\infty}F_h$.
 The next Proposition present some basic properties of $\Gamma$-limits which will be useful later on.

 \begin{prop} The following facts hold.
     \begin{itemize}
     \item[$\cdot$]
     For any $u\in X$ and for any sequence $(u_h)_h$ converging to $u$ in $X$, it holds that
         \begin{equation*}
    \Gamma-\liminf_{h\to\infty}F_h(u)\leq\liminf_{h\to\infty}F_h(u_h) \quad\mbox{and}\quad
    \Gamma-\limsup_{h\to\infty}F_h(u)\leq\limsup_{h\to\infty}F_h(u_h).
        \end{equation*}
    \item[$\cdot$] For any $u\in X$ there exist two sequences $(u_h)_h$ and $(v_h)_h$, converging to $u$ in $X$, which we call \emph{recovery sequences}, such that
         \begin{equation*}
        \Gamma-\liminf_{h\to\infty}F_h(u)=\liminf_{h\to\infty}F_h(u_h) \quad\mbox{and}\quad
        \Gamma-\limsup_{h\to\infty}F_h(u)=\limsup_{h\to\infty}F_h(v_h).
         \end{equation*}
         \end{itemize}
     \begin{itemize}
         \item[$\cdot$] For any $u\in X$ and for any sequence $(u_h)_h$ converging to $u$ in $X$, it holds that
         \begin{equation*}
             \Gamma-\lim_{h\to\infty}F_h(u)\leq\liminf_{h\to\infty}F_h(u_h);
         \end{equation*}
         \item[$\cdot$] For any $u\in X$ there exists a sequence $(u_h)_h$ converging to $u$ in $X$, which we call \emph{recovery sequence}, such that
         \begin{equation*}
             \Gamma-\lim_{h\to\infty}F_h(u)=\lim_{h\to\infty}F_h(u_h).
         \end{equation*}
     \end{itemize}
 \end{prop}

Beside the notion of $\Gamma$-convergence there is a related one, which is more suitable to deal with sequences of local functionals, usually known as \emph{$\overline{\Gamma}$-convergence}. If we have a sequence of increasing functionals such that $F_h:X\times\mA\scu\overline{\Ru}$ for any $h\in\mathbb{N}$, and we define
\begin{equation*}
    F'(\cdot,A):=\Gamma-\liminf_{h\to\infty}F_h(\cdot,A)\quad\text{and}\quad F''(\cdot,A):=\Gamma-\limsup_{h\to\infty}F_h(\cdot,A)
\end{equation*}
for any $A\in\mA$, we say that $F_h$ \emph{$\bar\Gamma$-converges} to a functional $F:X\times\mA\scu\bar\Ru$ if it holds that
\begin{equation*}
    F(\cdot,A)=\sup\{F'(\cdot,A')\,:\,A'\in\mA,\,A'\Subset A\}=\sup\{F''(\cdot,A')\,:\,A'\in\mA,\,A'\Subset A\}.
\end{equation*}
In other words we say that $(F_h)_h$ $\bar\Gamma-$converges to $F$ whenever the \emph{inner regular envelopes} of $F'$ and $F''$ coincide and are equal to $F$. It is easy to check (cf. \cite[Remark 16.3]{dalmaso}) that any $\bar\Gamma-$limit is increasing, inner regular and lower semicontinuous. In the sequel, when we will deal with $\Gamma$-convergence with respect to the strong topology of $L^p$ or with respect to the strong topology of $W_X^{1,p}$, we will refer respectively to  $\Gamma(L^p)-$convergence or $\Gamma(W_X^{1,p})-$convergence.

The notions of $\Gamma-$convergence and $\bar\Gamma$-convergence, as one could expect, are strongly related. Indeed, assume for instance that a sequence of increasing functionals $F_h:L^p(\Om)\times\mA\to [0,\infty]$ is such that 
 \begin{equation}\label{gammalim}
        F(\cdot,A)=\Gamma(L^p)-\lim_{k\to\infty}F_h(\cdot,A)
    \end{equation}
 for any $A\in\mA$ and for a suitable measure functional $F:L^p(\Om)\times\mA\to [0,\infty]$.
 Then $F$ is $L^p$-lower semicontinuous, since it is a $\Gamma$-limit (cf. \cite[Proposition 6.8]{dalmaso}, and also increasing and inner regular, since it is a non-negative measure (cf. \cite[Theorem 14.23]{dalmaso}). Therefore, \cite[Proposition 16.4]{dalmaso} allows to conclude that \begin{equation}\label{gammabar}
     F=\bar\Gamma(L^p)-\lim_{h\to\infty}F_h.
 \end{equation}
 The converse implication is usually more delicate because, in general, the $\bar\Gamma(L^p)-$limit is not a measure.
 Indeed, 
 even if the $\bar\Gamma-$limit is always increasing, inner regular and, even if superadditivity behaves usually well, there are examples (cf. \cite[Example 16.13]{dalmaso}) in which $F$ fails to be subadditive.
 Therefore, when dealing with this issues, it is practise to work within milder classes of local functionals. To this aim, the so-called \emph{uniform fundamental estimates} are introduced. These estimates, although depending in their definition on the chosen topological space, are usually sufficient conditions for the subadditivity of the $\bar\Gamma-$limit. To give an instance, we introduce here the standard notion of uniform fundamental estimate (cf. \cite[Definition 18.2]{dalmaso}) for functional defined on $\lp{\Om}\times\mA$.
 We recall here that, as we will deal also with functionals defined on $\anso{\Om}\times\mA$, in Section 4 we will need to slightly modify the current notion to guarantee a better compatibility with the new functional setting.
 \begin{deff}\label{fundes}
 If we have a class $\mathcal{F}$ of non-negative local functionals defined on $L^p(\Om)\times\mA$,
we say that $\mathcal{F}$ satisfies the \emph{uniform fundamental estimate on $L^p(\Om)$} if, for any $\varepsilon>0$ and for
any $A',A'',B \in\mA$, with $A' \Subset A''$, there exists a constant $M >\, 0$ such that for any $u,\,v\in L^p(\Om)$ and for any $F\in\mathcal{F}$, there exists a smooth cut-off function $\varphi$ between $A''$ and $A'$, such that
\begin{equation}\label{classicalfundest}
    \begin{split}
F\Big(\varphi u&+(1-\varphi)v, A'\cup B\Big)\le(1+\varepsilon)\Big(F(u,A'')+F(v,B)\Big)+\\
&+\varepsilon\Big( \|u\|_{L^p(S)}^p+\|v\|_{L^p(S)}^p+1\Big)+ M \|{u-v}\|_{L^p(S)},
\end{split}
\end{equation}
where $S = (A''\setminus A')\cap B$.
 \end{deff}

The following result, which can be found in \cite[Theorem 18.7]{dalmaso}, tells us that \eqref{gammabar} is sufficient to guarantee \eqref{gammalim}, provided that our sequence satisfies the uniform fundamental estimate and that some reasonable boundedness properties hold.
\begin{thm}\label{equivalenza}
    Let $F_h:\lp{\Om}\times\mA\scu[0,\infty]$ be a sequence of functionals for which there exists a functional $F:\lp{\Om}\times\mA\scu[0,\infty]$ such that \eqref{gammabar} holds.
    Assume in addition that $(F_h)_h$ satisfies the uniform fundamental estimate and that there exist constants $e_1\geq 1$ and $e_2\geq 0$, a non-negative increasing functional $G:\lp{\Om}\times\mA\scu[0,+\infty]$ and a finite measure $\mu$ on $\Om$ such that
    \begin{equation}\label{bound}
        G(u,A)\leq F_h(u,A)\leq e_1G(u,A) +e_2\|u\|^p_{\lp{A}}+\mu(A)
    \end{equation}
    for any $u\in\lp{\Om}$, $A\in\mA$ and $h\in\mathbb{N}$. Then \eqref{gammalim} holds.
\end{thm}

\section{$\Gamma$-compactness in $L^p$-Spaces}
In this section we prove a $\Gamma$-compactness result for a class of convex integral functionals defined on $L^p(\Om)$ with respect to the strong topology of $\lp{\Om}$. Our strategy is based on classical results of $\Gamma$-convergence in $L^p$ spaces and on the possibility to exploit the aforementioned integral representation result for convex local functionals presented in \cite[Theorem 2.3]{essverzpin}. 
First of all we introduce a large class of integral functionals for which some important properties, such as the uniform fundamental estimate introduced in Definition 
\ref{fundes}, are satisfied. Therefore we let $1<p<\infty$, and we fix $a\in\lu$ and constants $0<c_0\leq c_1$ and $c_2\geq 0$. We say that a functional $F:\lp{\Om}\times\mA\scu[0,\infty]$ belongs to $\claft$ if there exists a Carath\'eodory function
$f:\Om\times\Ru\times\Rm\scu[0,\infty]$ such that
\begin{equation}\label{coercbound}
    c_0|\eta|^p\leq f(x,u,\eta)\leq a(x)+c_1|\eta|^p+c_2|u|^p
\end{equation}
for any  $(u,\eta)\in\Ru\times\Rm$, for a.e. $x\in\Om$, and it holds that
\begin{equation}\label{708funzionaleduedef}
F(u,A)=
\displaystyle{\begin{cases}
\int_{A}f(x,u(x),Xu(x))\,{\rm d} x&\text{ if }A\in\mA,\,u\in W^{1,p}_X(A)\\
+\infty&\text{otherwise}
\end{cases}
\,.
}
\end{equation}
In particular, we say that $F\in\claf$ whenever $F\in\claft$ and it holds that 
\begin{equation}\label{convex}
    f(x,\cdot,\cdot) \text{ is convex for a.e. }x\in\Om.
\end{equation}

As announced, the main result of this section 
is the
$\Gamma$-compactness for the class of convex integral functionals.

\begin{thm}\label{709main1}
For any sequence $(F_h)_h\subseteq\claf$ there exists a subsequence $(F_{h_k})_k$ and a local functional $F\in\claf$ such that 
\begin{equation*}
    F(\cdot,A)=\Gamma(L^p)-\lim_{k\to+\infty}F_{h_k}(\cdot,A)\qquad\text{for any $A\in\mA$.}
\end{equation*}
\end{thm}

In order to prove the latter we first describe some properties of $\Gamma(L^p)$-limits within the class $\claft$. To this aim, we recall the following result, which can be found in \cite[Lemma 4.15]{maionepinafsk}.

\begin{prop}\label{708misura}
Let us define the functional $\Psi_p:\lp{\Om}\times\mA\scu[0,\infty]$ as
\begin{equation*}\label{708psidef}
\Psi_p(u,A):=
\displaystyle{\begin{cases}
\|Xu\|^p_{\lp{A}}&\text{ if }A\in\mA,\,u\in W^{1,p}_X(A)\\
+\infty&\text{ otherwise}
\end{cases}
\,.
}
\end{equation*}
Then $\Psi_p$ is a a $L^p$-lower semicontinuous measure.
\end{prop}

\begin{prop}\label{firststeplp}
For any sequence $(F_h)_h\subseteq\claft$ there exists a subsequence $(F_{h_k})_k$ and a functional $F:\lp{\Om}\times\mA\scu[0,\infty]$ such that 
\begin{itemize}
    \item $F$ is a measure
    \item $F$ is local
    \item $F$ is $L^p$-lower semicontinuous
    \item For any $u\in\anso{\Om}$ and $A\in\mA$ it holds that
    \begin{equation}\label{boundlp}
        \int_Ac_0|Xu(x)|^p {\rm d} x\leq F(u,A)\leq\int_Aa(x)+c_1|Xu(x)|^p +c_2|u(x)|^p
        {\rm d} x
    \end{equation}
\end{itemize}
and moreover it holds that
\begin{equation}\label{gammalimlp}
    F(\cdot,A)=\Gamma(L^p)-\lim_{k\to+\infty}F_{h_k}(\cdot,A)
\end{equation}
for any $A\in\mA$.
\end{prop}
\begin{proof}
The proof 
is based on general results of \cite{dalmaso}, indeed, 
according to \cite[Theorem 19.4]{dalmaso}, we introduce a suitable superclass of $\claft$. To this aim, we say that a functional $F:\lp{\Om}\times\mA\scu[0,\infty]$ belongs to $\clam$ if $F$ is a measure and if
there exist $d_1\geq 1$, $d_2,d_3,d_4\geq 0$, a finite measure $\mu$, independent of $F$, and a measure $G:\lp{\Om}\times\mA\scu[0,\infty]$, which may depend on $F$, such that
 \begin{equation}\label{708ufe1}
    G(u,A)\leq F(u,A)\leq d_1 G(u,A)+d_2\|u\|_{\lp{A}}+\mu(A)
\end{equation}
and
\begin{equation}\label{708ufe2}
    G(\varphi u+(1-\varphi)v,A)\leq d_4(G(u,A)+G(v,A))+d_3d_4(\max|D\varphi|^p)\|u-v\|_{L^p(A)}+\mu(A),
\end{equation}
for any $u,v\in\lp{\Om}$, $A\in\mA$ and $\varphi\in C^\infty_c(\Om)$ such that $0\leq\varphi\leq 1$.
We are going to show that $\claft\subseteq\clam$. For this purpose, let us define $\mu:\mathbb{B}\scu[0,+\infty]$ as 
$$\mu(B):=\int_B|a(x)| \, {\rm d} x.$$
Then $\mu$ is a finite measure on $\Om$. Moreover, thanks to Proposition \ref{708misura}, the non-negative local functional $G:\lp{\Om}\times\mA\scu[0,+\infty]$ defined as
$$G(u,A):=c_0\Psi_p(u,A)\qquad\text{for any $u\in\lp{\Om},\,A\in\mA$}$$
is a measure. Let us show \eqref{708ufe1}. Let us set $d_1:=\frac{c_1}{c_0}$ and $d_2:=c_2$. If $A\in\mA$ and $u\notin\anso{A}$, the estimate is trivial, while if $u\in\anso{A}$, it follows from the definition of $\cla$. So we are left to show \eqref{708ufe2}. Fix then $A\in\mA$. If either $u\notin\anso{A}$ or $v\notin\anso{A}$ the estimate is trivial. Hence assume that $u,v\in\anso{A}$ and take $\varphi\in C^\infty_c(\Om)$ such that $0\leq\varphi\leq 1$. Then, recalling Proposition \ref{503soinanso}, Proposition \ref{708leibniz}, the fact that $\eta\mapsto|\eta|^p$ is convex on $\Rm$, and setting $$C:=\max\{\|c_{j,i}\|_{\infty}\,:\,j=1,\ldots,m,\,i=1,\ldots,n\}$$
it follows that $0<C<\infty$ and
\begin{equation*}
    \begin{split}
    G(\varphi u+(1-\varphi)v,A)&=c_0\int_A \lvert X\varphi(u-v)+\varphi Xu+(1-\varphi)Xv \rvert^p \, {\rm d} x\\
    &= c_02^p\int_A\left|\frac{X\varphi(u-v)}{2}+\frac{\varphi Xu+(1-\varphi)Xv}{2}\right|^p \, {\rm d} x\\
    &\leq c_02^{p-1}\int_A  \lvert X\varphi(u-v) \rvert^p \, {\rm d} x +c_02^{p-1}\int_A|\varphi Xu+(1-\varphi)Xv|^p \, {\rm d} x\\
    &\leq c_02^{p-1}\int_A |X\varphi(u-v)|^p \, {\rm d} x +2^{p-1}(G(u,A)+G(v,A))\\
    &\leq c_02^{p-1}(C\sqrt{m})^p(\max|Du|^p) \|u-v\|_{L^p(A)}+2^{p-1}(G(u,A)+G(v,A)).\\
    \end{split}
\end{equation*}
Thus \eqref{708ufe2} follows. Hence $\claft\subseteq\clam$. Therefore, thanks to \cite[Theorem 19.5]{dalmaso}, there exist a subsequence of $(F_h)_h$, still denoted by $(F_h)_h$, and a $L^p$-lower semicontinuous functional $F\in\clam$ such that $(F_h)_h$ $\overline{\Gamma}(L^p)$-converges to $F$. In particular $F$ is a measure. By \cite[Proposition 16.4]{dalmaso} and \cite[Proposition 16.15]{dalmaso}, $F$ is also local. Furthermore, by Proposition \ref{708misura} $G$ is a $L^p$-lower semicontinuous measure and, since $(F_h)_h$ satisfies the uniform fundamental estimate on $L^p(\Om)$ according to \cite[Theorem 19.4]{dalmaso}, we can apply Theorem \ref{equivalenza} to conclude that \eqref{gammalimlp} holds. Finally, we show that $F$ satisfies \eqref{boundlp}.
Let us fix $A\in\mA$ and $u\in\anso{\Om}$, and a sequence $(u_h)_h$ such that 
\begin{equation}\label{709approxseq2}
    F(u,A)=\lim_{h\to+\infty}F_h(u_h,A).
\end{equation}
Arguing as above we can assume that $(u_h)_h \subset \anso{A}$. Therefore, thanks to \eqref{709approxseq2} and Proposition \ref{708misura}, it follows that
\begin{equation}\label{709stimabasso}
    c_0\int_A \lvert Xu \rvert^p \,{\rm d} x \leq\liminf_{h\to+\infty}\int_A \lvert Xu_h \rvert^p
    \,{\rm d} x
    \leq\liminf_{h\to+\infty}F_h(u_h,A)=F(u,A),
\end{equation}
and so the first inequality follows. 
Finally we have that
\begin{equation*}
\begin{split}
    F(u,A)&\leq\liminf_{h\to+\infty}F_h(u,A)\leq\liminf_{h\to+\infty}\int_A a(x)+c_1
    \lvert Xu \rvert^p+c_2 \lvert u \rvert^p \,{\rm d} x\\
    &=\int_A a(x)+c_1 \lvert Xu \rvert^p+c_2 \lvert u \rvert^p
    \,{\rm d} x.
\end{split}
\end{equation*}

This proves the thesis. 
\end{proof}

In order to represent the $\Gamma$-limit achieved in the previous proposition in an integral form, we wish to exploit \cite[Theorem 2.3]{essverzpin}. Following a remark that the authors of \cite{essverzpin} made in their introduction, we present here a slight variant which is more suitable to our purposes.

\begin{thm}\label{2411modiofmythm}
 Let $F:\lp{\Om}\times\mA\scu[0,\infty]$ be such that:
	\begin{itemize}
		\item [(i)]$F$ is a measure;
		\item [(ii)]$F$ is local;
		\item [(iii)]$F$ is convex on $\anso{\Om}$;
		\item [(iv)]$F$ satisfies \eqref{boundlp}.
	\end{itemize}
	Then there exists a Carath\'eodory function $f:\Om\times\Ru\times\Rm\to[0,\infty]$ which satisfies \eqref{coercbound} and \eqref{convex}, and such that
\begin{equation}\label{repr}
    F(u,A)=\int_{A}f(x,u(x),Xu(x)){\rm d} x
\end{equation}
for any $A\in\mA$ and for any $u\in\anso{A}$.
\end{thm}
\begin{proof}
We point out that in \cite{essverzpin} the authors did not take into account the possible equivalence between the bound from below of the Lagrangian and the bound from below of the functional, as the latter is actually not necessary to represent an abstract convex local functional in integral form. On the other hand, it is clear from the proofs in \cite{essverzpin} that such an equivalence is trivial, and so, in the current paper, we take it for granted.
From \cite{essverzpin} we know that there exists a a Carath\'eodory function $f:\Om\times\Ru\times\Rm\to[0,\infty]$ which satisfies \eqref{coercbound} and \eqref{convex}, and such that
\begin{equation*}
    F(u,A)=\int_Af(x,u(x),Xu(x)){\rm d} x
\end{equation*}
for any $A\in\mA$ and for any $u\in\anso{\Om}$. Fix now $A\in\mA$, $A'\in\mA_0$ with $A'\Subset A$ and $ u\in\lp{\Om}\cap\anso{A}$, and let $v:=\varphi u$, where $\varphi$ is a smooth cut-off function between $A'$ and $A$. Then clearly $v\in\anso{\Om}$ and $v|_{A'}=u$. As $F$ is local, it follows that 
\begin{equation*}
    F(u,A')=F(v,A')=\int_{A'}f(x,v(x),Xv(x)){\rm d} x=\int_{A'}f(x,u(x),Xu(x)){\rm d} x.
\end{equation*}
Since $F$ is a measure, it is in particular inner regular, and so we conclude that \eqref{repr} holds.
\end{proof}

We are now ready to prove Theorem \ref{709main1}.
\begin{proof}[Proof of Theorem \ref{709main1}]
As $F\in\claf$, thanks to Proposition \ref{firststeplp} there exists a functional $F:\lp{\Om}\times\mA\scu[0,+\infty]$ which is a measure, local, satisfies \eqref{boundlp} and such that \eqref{gammalimlp} holds. Let us show that $F$ is convex on $\anso{\Om}$. 
Fix then $A\in\mA$ and take $t\in(0,1)$ and $u,v\in\anso{\Om}$. Let $(u_h)_h$ and $(v_h)_h$ be two sequences converging respectively to $u$ and $v$ in $L^p(\Om)$ and such that 
\begin{equation}\label{709approxseq}
    F(u,A)=\lim_{h\to+\infty}F_h(u_h,A),\qquad F(v,A)=\lim_{h\to+\infty}F_h(v_h,A).
\end{equation}
Since $F(u,A)$ and $F(v,A)$ are finite we can assume that the 
sequences $(u_h)_h,(v_h)_h$ belong to $\anso{A}$. Therefore, since each $F_h(\cdot,A)$ is convex on $\anso{A}$, recalling \eqref{709approxseq} and the fact that $(tu_h+(1-t)v_h)_h$ converges to $tu+(1-t)v$ in $L^p(\Om)$, it follows that
\begin{equation*}
    \begin{split}
        F(tu+(1-t)v,A)&\leq\liminf_{h\to+\infty}F_h(tu_h+(1-t)v_h,A)\\
        &\leq\liminf_{h\to+\infty}(tF_h(u_h,A)+(1-t)F_h(v_h,A))\\
        &=t\lim_{h\to+\infty}F_h(u_h,A)+(1-t)\lim_{h\to+\infty}F_h(v_h,A)\\
        &=tF(u,A)+(1-t)F(v,A).
    \end{split}
\end{equation*}
Therefore we are in position to apply Theorem \ref{2411modiofmythm}. Finally, we notice that if $A\in\mA$ and $u\in\lp{\Om}\setminus\anso{A}$, arguing as in \eqref{709stimabasso} we conclude that $+\infty=c_0\Psi_p(u,A)\leq F(u,A)$, which implies that
\begin{equation*}
    \{u\in\lp{\Om}\,:\,F(u,A)<+\infty\}=\anso{A},
\end{equation*}
and so the thesis follows.
\end{proof}

\section{$\Gamma$-compactness in $W_X^{1,p}$}
In this section we show two $\Gamma$-compactness results for suitable classes of integral functionals defined on $\anso{\Om}$ and with respect to the strong topology of $\anso{\Om}$. Working in this new framework has surely some advantages. For instance we do not have to assume any coercivity assumptions on the sequence of Lagrangians, and we can allow the case $p=1$, since, among the other things, Proposition \ref{708misura} is not needed anymore. Therefore, throughout this section, we let $1\leq p<+\infty$ and, as in the previous one, we fix $a\in\lu, \ c_1,c_2\geq 0$. We say that a functional $F:\anso{\Om}\times\mA\scu[0,\infty]$ belongs to $\clagt$ if there exists a Carath\'eodory function
$f:\Om\times\Ru\times\Rm\scu[0,\infty]$ such that
\begin{equation}\label{onlybounded}
    f(x,u,\eta)\leq a(x)+c_1|\eta|^p+c_2|u|^p
\end{equation}
for any  $(u,\eta)\in\Ru\times\Rm$ and for a.e. $x\in\Om$, and it holds that
\begin{equation*}\label{708funzionaleduedef}
F(u,A)=\int_{A}f(x,u(x),Xu(x))dx
\end{equation*}
for any $A\in\mA$ and any $u\in\anso{\Om}$. 
Similarly to the previous section, we will show that this large class of functionals satisfies many nice properties, among which a suitable notion of uniform fundamental estimate that will be introduced below. However, this class is too general to hope to achieve $\Gamma$-compactness. Therefore we define two sub-classes which will be shown to be $\Gamma$-compact. For the first case we consider the sub-class of the convex functionals belonging to $\clagt$, i.e. we say that $F\in\clag$ whenever $F\in\clagt$ and
$$f(x,\cdot,\cdot) \text{ is convex for a.e. }x\in\Om.$$
In the second case we want to drop the convexity assumption. To this aim, we introduce a notion of strong condition which is strongly inspired by the classical one introduced in \cite{buttazzodalmaso2}.
\begin{deff}\label{nuovastrongdef}
We say that $\omega=(\omega_s)_{s\geq 0}$ is a \emph{family of locally integrable moduli of continuity} if $\omega_s:\Om\times[0,+\infty)\scu[0,+\infty)$ and 
	\begin{equation}\label{incruno}
		r\mapsto\omega(x,r)\text{ is increasing, continuous and }\omega(x,0)=0 
	\end{equation}
	for a.e. $x\in\Om$ and for any $s\geq0$,
	\begin{equation}\label{incrdue}
		s\mapsto\omega_s(x,r)\text{ is increasing and continuous} 
	\end{equation}
	for a.e. $x\in\Om$ and for any $r\geq 0$, and
	\begin{equation*}\label{2711modulus2}
		x\mapsto\omega_s(x,r)\in L^1_{loc}(\Om) \quad \text{for any } r,s\geq 0.
	\end{equation*}
Moreover we say that a functional $F:\anso{\Om}\times\mA\scu[0,\infty]$ satisfies the \emph{strong condition $(\omega X)$ with respect to $\omega$} if there exists a family $\omega=(\omega_s)_{s\geq 0}$ of locally integrable moduli of continuity such that
		\begin{equation}\label{2811strongprop}
			|F(v,A')-F(u,A')|\leq\int_{A'}\omega_s(x,r)\, {\rm d} x
		\end{equation}
		for any $s\geq 0$, $A'\in\mA_0$, $r\geq 0$, $u,v\in \anso{\Om}$ such that 
		\begin{align*}
		&|u(x)|,|v(x)|,|Xu(x)|,|Xv(x)|\leq s\\
		&|u(x)-v(x)|,|Xu(x)-Xv(x)|\leq r
		\end{align*}
		for a.e. $x\in A'$.
	\end{deff}	
This new notion seems to be more flexible and to fit better with our non-Euclidean setting, and allows to deal with more general classes of functions. On the other hand, it is quite easy to see that our condition is stronger than the one introduced in \cite{buttazzodalmaso2}, and so all the integral representation results proved in \cite{buttazzodalmaso2,essverzpin} remain valid.
Moreover, we point out that our family of moduli of continuity, unlike in \cite{buttazzodalmaso2}, is indexed over a continuous set, and the assumption on the behaviour of $s\mapsto \omega_s(x,r)$ is completely new. Nevertheless we will see in a while that, at least when dealing with integral functionals, this new requirement is quite natural.
Indeed the following fact holds.
\begin{prop}\label{effeisstrong}
Let $F\in\clagt$.
Then $F$ satisfies the strong condition $(\omega X)$.
\end{prop}
\begin{proof}
	This proof is based on the proof of \cite[Lemma 2.5]{buttazzodalmaso2}. Since $f$ is Carath\'eodory, then the set $\Om':=\{x\in\Om\,:\,(u,\xi)\mapsto f(x,u,\xi)\text{ is continuous}\}$ satisfies $|\Om'|=|\Om|$.
	For any $s,r\geq 0$, set $E^s_r\subseteq\Ru\times\Ru\times\Rm\times\Rm$ as $$E^s_r:=\{(u,v,\xi,\eta)\,:\,|u|,|v|,|\xi|,|\eta|\leq s, |u-v|,|\xi-\eta|\leq r\}$$
	and the function
	\begin{equation*}\label{2911modulusdef}
		\omega_s(x,r):=
		\begin{cases}\,\sup\{|f(x,u,\xi)-f(x,v,\eta)|\,:\,(u,v,\xi,\eta)\in E^s_r\}&\text{ if }x\in\Om',\\
			0&\text{ otherwise.}
		\end{cases}
	\end{equation*}
	We show that $(\omega_s)_{s\geq 0}$ is a family of locally integrable moduli of continuity. Let us fix then $s,r\geq 0$: since $(u,\xi)\mapsto f(x,u,\xi)$ is continuous for almost every $x\in\Om$, then the supremum in the definition of $\omega_s$ can be taken over a countable subset of $E^k_\epsilon$. Since for any $(u,v,\xi,\eta)$ the function $x\mapsto|f(x,u,\xi)-f(x,v,\eta)|$ is measurable, then $\omega_s(\cdot,r)$ is 
	measurable. Moreover, thanks to \eqref{onlybounded}, it follows that, for any $(u,v,\xi,\eta)\in E^k_\epsilon$,
	\begin{equation*}
		\begin{split}
			|f(x,u,\xi)-f(x,v,\eta)|&\leq2|a(x)|+c_1|\xi|^p+c_1|\eta|^p+c_2|u|^p+c_2|v|^p\\
			&\leq2|a(x)|+4s(c_1+c_2).
		\end{split}	
	\end{equation*}
	Since the right side does not depend on $(u,v,\xi,\eta)\in E^s_r$, we conclude that
	$$\omega_s(x,r)\leq2|a(x)|+4s(c_1+c_2).$$
	Hence $\omega_k(\cdot,\epsilon)\in L^1_{loc}(\Om)$. Fix now $x\in\Om'$ and $s\geq 0$.
	Since $E^s_r\subseteq E^s_t$ for any $r\leq t$, then $\omega_s(x,\cdot)$ is increasing, $\omega_k(x,0)=0$
	and the continuity follows from the continuity of $f(\cdot,u,\xi)$.
	Finally, taking $x\in\Om'$ and $r\geq 0$ we have again that $E^s_r\subseteq E^t_r$ for any $r\leq t$, hence $s\mapsto\omega_s(x,r)$ is increasing. Once more, from the continuity of $f(\cdot,u,\xi)$ we conclude that $s\mapsto\omega_s(x,r)$ is continuous. Then $(\omega_s)_s$ is a family of locally integrable moduli of continuity. 
	It is straightforward to check that $F$ satisfies the strong condition $(\omega X)$ with respect to $(\omega_s)_{s\geq 0}$.
\end{proof}
On the other hand, if $(F_h)_h\subseteq\clagt$, even if each $F_h$ satisfies the strong condition $(\omega X)$, in general the family of moduli of continuity strongly depends on $h$. Therefore we introduce suitable subclasses of $\clagt$ which present uniformity in the choice of the family of moduli of continuity. Hence, if a family $\omega=(\omega_s)_{s\geq 0}$ is fixed, we say that a functional $F:\anso{\Om}\times\mA\scu[0,\infty]$ belongs to $\clafto$ if $F\in\clagt$ and it satisfies the strong condition $(\omega X)$ with respect to $\omega$.

\begin{rem}
Let $(F_h)_h\subseteq\clagt$ be such that there exists $K\in\lul$ such that
\begin{equation}\label{lipschitz}
    |f_h(x,u,\xi)-f_h(x,v,\eta)|\leq |K(x)|(|u-v|+|\xi-\eta|)
\end{equation}
for any $u,v\in\Ru$, $\xi,\eta\in\Rm$ and $h\in\mathbb{N}$. If for any $s,r\geq 0$ we define $E^s_r$ as in Proposition \ref{effeisstrong}
	and
	\begin{equation*}\label{2911modulusdef}
		\Tilde{\omega}_s(x,r):=|K(x)|\sup\{(|u-v|+|\xi-\eta|)\,:\,(u,v,\xi,\eta)\in E^s_r\}
		,
	\end{equation*}
	then it is easy to see that $(F_h)_h$ belongs to $\claftot$.
\end{rem}
We are ready now to state the two main results of this section.
\begin{thm}\label{mainsobolev2}
For any sequence $(F_h)_h\subseteq\clag$ there exists a subsequence $(F_{h_k})_k$ and a functional $F\in\clag$ such that 
\begin{equation*}
    F(\cdot,A)=\Gamma(W_X^{1,p})-\lim_{k\to+\infty}F_{h_k}(\cdot,A)\qquad\text{for any $A\in\mA$.}
\end{equation*}
\end{thm}
\begin{thm}\label{mainsobolev1}
For any sequence $(F_h)_h\subseteq\clafto$ there exists a subsequence $(F_{h_k})_k$ and a functional $F\in\clafto$ such that 
\begin{equation*}
    F(\cdot,A)=\Gamma(W_X^{1,p})-\lim_{k\to+\infty}F_{h_k}(\cdot,A)\qquad\text{for any $A\in\mA$.}
\end{equation*}
\end{thm}

As already said, one of the key step for the proof of these results is introducing a suitable notion of uniform fundamental estimate. Therefore, inspired by the classical notion stated in \cite{dalmaso}, we give the following definition.
\begin{deff}
Let $\mathcal{F}$ be a class of non-negative local functionals defined on $\anso{\Om}\times\mA$.
We say that $\mathcal{F}$ satisfies the \emph{uniform fundamental estimate on $\anso{\Om}$} if, for any $\varepsilon>0$ and for
any $A',A'',B \in\mA$, with $A' \Subset A''$, there exists a constant $M >\, 0$ and a finite family $\{\varphi_1,\ldots,\varphi_k\}$ of smooth cut-off functions between $A'$ and $A''$ such that for any $u,\,v\in \anso{\Om}$ and for any $F\in\mathcal{F}$, we can choose $\varphi\in\{\varphi_1,\ldots,\varphi_k\}$ such that
\[
\begin{split}
F\Big(\varphi u&+(1-\varphi)v, A'\cup B\Big)\le\Big(F(u,A'')+F(v,B)\Big)+\\
&+\varepsilon\Big( \|u\|_{\anso{S}}^p+\|v\|_{\anso{S}}^p+1\Big)+ M \|{u-v}\|_{L^p(S)},
\end{split}
\]
where $S = (A''\setminus A')\cap B$.
\end{deff}
Let us point out the differences between the two definitions. From one hand, this estimate is stronger, since it requires that the choice of the cut-off function must be done among a finite family of candidates which depends only on $\eps,A',A''$ and $B$. This requirement, as we will see, is crucial to guarantee a uniform estimate for the $X$-gradients of the test functions. However, we replace some of the $L^p$ norms on the right hand side of \eqref{classicalfundest} with $W_X^{1,p}$-norms, thus weakening some of the requirements. This choice, as we will see, is crucial to avoid the coercivity assumptions on the Lagrangians. The following results and their proofs are respectively the counterparts of \cite[Proposition 19.1]{dalmaso} and \cite[Proposition 18.3]{dalmaso}.

\begin{prop}\label{nuovafundest}
$\clagt$ satisfies the uniform fundamental estimate on $\anso{\Om}$.
\end{prop}
\begin{proof}
 
Let us set $d_1:=c_1$, $d_2:=c_2$ and $d_4:=2^{p-1}$ and $\sigma(C):=\int_C |a(x)|dx$ for any $C\in\mathcal{B}$. Fix $\eps>0$, $B\in\mA$ and $A',A''\in\mA$ with $A'\Subset A''$. Choose $A\in\mA$ with $A'\Subset A\Subset A''$ and $k\in\mathbb{N}$ with
\begin{equation*}
    \max\left\{\frac{d_1+d_2d_4}{k},\frac{\sigma(A\setminus\overline{A'})}{k}\right\}<\eps.
\end{equation*}
Moreover, choose open sets $A_1,\ldots,A_{k+1}$ such that $A'\Subset A_1\Subset\,\ldots\,\Subset A_{k+1}\Subset A$, and, for any $i=1,\ldots,k$ take a smooth cut-off function $\varphi_i$ between $A_i$ and $A_{i+1}$. Finally, set
\begin{equation*}
    M:=\frac{d_1d_4}{k}\max_{1\leq i\leq k}\max_{x\in\Om}|X\varphi_i(x)|^p.
\end{equation*}
Let $F\in\clagt$ and $u,v\in\anso{\Om}$. Then, for any $i=1,\ldots,k$, from the choice of $\varphi_i$ it follows that
\begin{equation}\label{stimatrepezzi}
    F(\varphi_i u+(1-\varphi_i)v,A'\cup B)\leq F(u,(A'\cup B)\cap \overline{A_i})+F(v,B\setminus A_{i+1})+F(\varphi_i u+(1-\varphi_i)v,S_i),
\end{equation}
where $S_i:=B\cap (A_{i+1}\setminus\overline{A_i})$. Setting $I_i \coloneqq F(\varphi_i u+(1-\varphi_i)v,S_i),$ 
from the bound 
on the Lagrangian and arguing as in the proof of Proposition \ref{firststeplp}, we get that 
\begin{equation*}
    \begin{split}
        I_i&\leq d_1\int_{S_i}|X(\varphi_i u+(1-\varphi_i)v)|^p {\rm d}x+d_2\int_{S_i}|\varphi_i u+(1-\varphi_i)v|^p {\rm d}x+\sigma(S_i)\\
        &=d_1\int_{S_i}|uX\varphi_i +\varphi_i Xu-vX\varphi_i+(1-\varphi_i)Xv)|^p {\rm d}x+d_2\int_{S_i}| u|^p {\rm d}x +d_2\int_{S_i}|v|^p {\rm d}x+\sigma(S_i)\\
        &=d_1\int_{S_i}|(\varphi_i Xu+(1-\varphi_i)Xv)+X\varphi_i(u-v)|^p {\rm d}x+d_2\int_{S_i}(| u|^p + |v|^p ){\rm d}x +\sigma(S_i) \\
        &\leq d_1d_4\left[\int_{S_i}|\varphi_i Xu+(1-\varphi_i)Xv|^p+\int_{S_i}|X\varphi_i|^p|u-v|^p {\rm d}x\right]+d_2 \int_{S_i}(| u|^p + |v|^p ){\rm d}x +\sigma(S_i)\\
        &\leq d_1d_4\left[\int_{S_i}| Xu|^p {\rm d}x+\int_{S_i}|Xv|^p {\rm d}x\right]+kM\int_{S_i}|u-v|^p {\rm d}x+d_2 \int_{S_i}(| u|^p + |v|^p ){\rm d}x+\sigma(S_i)\\
        &\leq (d_2+d_1d_4)\left(\|u\|^p_{\anso{S_i}}+\|v\|^p_{\anso{S_i}}\right)+kM\|u-v\|^p_{\lp{S_i}}+\sigma(S_i).
    \end{split}
\end{equation*}
Noticing that $\sigma$ is a measure and that 
$$S_1\cup\,\ldots\,\cup S_k\subseteq (A\setminus\overline{A'})\cap B\subseteq S,$$ 
and recalling the choice of $k$, it follows that
\begin{equation}\label{stimafinalefund}
\begin{split}
    \min_{1\leq i\leq k}I_i&\leq\frac{1}{k}\sum_{i=1}^k I_k\leq\frac{d_2+d_1d_4}{k}\left(\|u\|^p_{\anso{S}}+\|v\|^p_{\anso{S}}\right)+M\|u-v\|^p_{\lp{S}}+\frac{\sigma(A\setminus\overline{A'})}{k}\\
    &\leq\eps\left(\|u\|^p_{\anso{S}}+\|v\|^p_{\anso{S}}+1\right)+M\|u-v\|^p_{\lp{S}}.
\end{split}
\end{equation}
Therefore, if $\varphi_i\in\{\varphi_1,\ldots,\varphi_k\}$ is chosen to realize the minimum, observing that $F$ is a measure, $(A'\cup B)\cap\overline{A_i}\subseteq A''$ and $B\setminus A_{i+1}\subseteq B$, thanks to \eqref{stimatrepezzi} and \eqref{stimafinalefund} the thesis follows.
\end{proof}

\begin{prop}\label{quasisubprop}
Let $(F_h)_h\in\clagt$. Then it holds that
\begin{equation}\label{quasisub}
    F''(u,A'\cup B)\leq F''(u,A'')+F''(u,B)
\end{equation}
for any $u\in\anso{\Om}$, $B\in\mA$ and $A',A''\in\mA$ with $A'\Subset A''$.
\end{prop}
\begin{proof}
Let $u,A',A'',B$ as above
fix $\eps>0$, and let $(u_h)_h,(v_h)_h\subseteq\anso{\Om}$ be two recovery sequences for $u$ with respect to $F''(\cdot,A'')$ and $F''(\cdot,B)$ respectively. From Proposition \ref{nuovafundest} we know that $(F_h)_h$ satisfies the uniform fundamental estimate on $\anso{\Om}$. Therefore there exists $M>0$ and a finite family $\{\varphi^1,\ldots,\varphi^k\}$ of smooth cut-off functions between $A'$ and $A''$, depending only on $\eps,A',A''$ and $B$, and a sequence $(\varphi_h)_h\subseteq\{\varphi^1,\ldots,\varphi^k\}$, such that
\begin{equation}\label{fundconacca}
   \begin{split}
F_h\Big(\varphi_h u_h+(1-\varphi_h)v_h, A'\cup B\Big)&\le\Big(F_h(u_h,A'')+F_h(v_h,B)\Big)+\\
&+\varepsilon\Big( \|u_h\|_{\anso{S}}^p+\|v_h\|_{\anso{S}}^p+1\Big)+ M \|{u_h-v_h}\|_{L^p(S)},
\end{split} 
\end{equation}
where $S = (A''\setminus A')\cap B$.
Let us define $w_h:=\varphi_h u_h+(1-\varphi_h)v_h$. Then it follows that
\begin{equation*}
    \|w_h-u\|_{\lp{\Om}}=\|\varphi_h(u_h-v_h)\|_{\lp{\Om}}+\|v_h-u\|_{\lp{\Om}}\leq\|u_h-v_h\|_{\lp{\Om}}+\|v_h-u\|_{\lp{\Om}},
\end{equation*}
and moreover
\begin{equation*}
\begin{split}
    \|Xw_h-Xu\|_{\lp{\Om}}&=\|X\varphi_h\cdot u_h+\varphi_hXu_h-X\varphi_h\cdot v_h+(1-\varphi_h)Xv_h-Xu\|_{\lp{\Om}}\\
    &\leq\|X\varphi_h(u_h-v_h)\|_{\lp{\Om}}+\|\varphi_h(Xu_h-Xv_h)\|_{\lp{\Om}}+\|Xv_h-Xu\|_{\lp{\Om}}\\
    &\leq\max_{1\leq i\leq k}\||X\varphi^k|^p\|_{\infty}\cdot\|u_h-v_h\|_{\lp{\Om}}+\|Xu_h-Xv_h\|_{\lp{\Om}}+\|Xv_h-Xu\|_{\lp{\Om}}.
\end{split}
\end{equation*}
Therefore we conclude that $w_h$ converges to $u\in\anso{\Om}$. This fact, the choices of $u_h$ and $v_h$ and \eqref{fundconacca} allow to conclude that
\begin{equation*}
\begin{split}
    F''(u,A'\cup B)&\leq\limsup_{h\to\infty}F''(w_h,A'\cup B)\\
    &\leq\limsup_{h\to\infty}F''(u_h,A'')+\limsup_{h\to\infty}F''(v_h,B)\\
    &+\varepsilon\Big( \|u\|_{\anso{S}}^p+\|v\|_{\anso{S}}^p+1\Big)\\
    &=F''(u,A'')+F''(u,B)+\varepsilon\Big( \|u\|_{\anso{S}}^p+\|v\|_{\anso{S}}^p+1\Big).
\end{split}
\end{equation*}
Being $\eps$ arbitrary, the thesis follows.
\end{proof}
We are ready to complete \textbf{Step 1} of our general scheme.
\begin{prop}\label{firststepsobolev}
For any sequence $(F_h)_h\subseteq\clagt$ there exists a subsequence $(F_{h_k})_k$ and a functional $F:\anso{\Om}\times\mA\scu[0,\infty]$ such that 
\begin{itemize}
    \item $F$ is a measure
    \item $F$ is local
    \item $F$ is $W_X^{1,p}-$lower semicontinuous
    \item  for any $u\in\anso{\Om}$ and $A\in\mA$ it holds that
    \begin{equation}\label{boundsobolev}
        F(u,A)\leq\int_Aa(x)+c_1|Xu(x)|^p+c_2|u(x)|^p {\rm d} x
    \end{equation}
\end{itemize}
and moreover we have that
\begin{equation}\label{gammalimsobolev}
    F(\cdot,A)=\Gamma(W_X^{1,p})-\lim_{k\to+\infty}F_{h_k}(\cdot,A)
\end{equation}
for any $A\in\mA$.
\end{prop}
\begin{proof}
Since $(\anso{\Om},\|\cdot\|_{\anso{\Om}})$ is a metric space, by \cite[Theorem 16.9]{dalmaso} we know that, up to a subsequence, $(F_h)_h\,$ $\bar\Gamma(W_X^{1,p})$-converges to a functional $F:\anso{\Om}\times\mA\scu\overline{\Ru}$. Being $F$ a $\bar\Gamma$-limit, we know from \cite[Remark 16.3]{dalmaso} that $F$ is increasing, inner regular and $W_X^{1,p}-$ lower semicontinuous. Moreover, thanks to \cite[Proposition 16.12]{dalmaso}, we know that $F$ is superadditive.
Let us show that $F$ is non-negative. Indeed, fix $A\in\mA$ and $u\in\anso{\Om}$, then we know that
\begin{equation*}\label{innerenvelope}
    F(u,A)=\sup\bigl\{\inf\{\liminf_{h\to\infty}F_h(u_h,A')\,:\,u_h\to u\text{ in }\anso{\Om}\}\,:\,A'\in\mA,\,A'\Subset A\bigr\}.
\end{equation*}
As each $F_h(u_h,A')$ is non-negative, then $F(u,A)\geq 0$.
Moreover, in the same way we can see that $F(u,\emptyset)=0$ for any $u\in\anso{\Om}.$
Now, adapting the proof of \cite[Proposition 16.15]{dalmaso}, we show that $F$ is local.
Let us fix $A\in\mA$ and $u,v\in\anso{\Om}$ coinciding a.e. on $A$. Fix $A'\Subset A$, take a smooth cut-off function $\varphi$ between $A'$ and $A$ and let $(u_h)_h\subseteq\anso{\Om}$ be a recovery sequence for $u$ with respect to $F'(\cdot,A')$. We define a new sequence $(v_h)_h$ requiring that
$$v_h:=\varphi u_h+(1-\varphi)v.$$
It is clear that
\begin{equation*}
    \|v_h-v\|_{\lp{\Om}}=\|\varphi(u_h-v)\|_{\lp{\Om}}=\|\varphi(u_h-v)\|_{\lp{A}}\leq\|u_h-u\|_{\lp{A}},
\end{equation*}
and moreover
\begin{equation*}
\begin{split}
    \|Xv_h-Xv\|_{\lp{\Om}}&=\|X\varphi(u_h-v)+\varphi(Xu_h-Xv)\|_{\lp{\Om}}\\
    &\leq\|X\varphi(u_h-v)\|_{\lp{A}}+\|\varphi(Xu_h-Xv)\|_{\lp{A}}\\
    &\leq \||X\varphi|^p\|_\infty\|u_h-u\|_{\lp{A}}+\|Xu_h-Xu\|_{\lp{A}}.
\end{split}
\end{equation*}
Therefore we have that $v_h$ converges to $v$ in $\anso{\Om}$. As each $F_h$ is local and $u_h=v_h$ on $A'$, we conclude that
\begin{equation*}
\begin{split}
    F'(v,A')\leq\liminf_{h\to\infty}F_h(v_h,A')=\liminf_{h\to\infty}F_h(u_h,A')=F'(u,A').
\end{split}
\end{equation*}
As the converse inequality can be proved exchanging the roles of $u$ and $v$, we conclude that $F'(u,A')=F'(v,A')$. Finally, being $A'\Subset A$ arbitrary and recalling the definition of a $\bar\Gamma-$limit, we conclude that $F$ is local. Moreover, thanks to Proposition \ref{quasisubprop}, we can repeat essentially the same steps of the proof of \cite[Proposition 18.4]{dalmaso} and achieve that $F$ is subadditive. Notice that, thanks to \cite[Theorem 14.23]{dalmaso} and the previous steps, this suffices to conclude that $F$ is a measure.
If we define now $G :\anso{\Om}\times\mA \rightarrow [0,+\infty]$ as
$$G(u,A):=\int_A a(x)+c_2|u|^p+c_1|Xu|^p$$
for any $u\in\anso{\Om}$ and for any $A\in\mA$, it is clear that $G$ is a measure and that, thanks to our hypotheses,
$F_h\leq G$ for any $h\in\mathbb{N}$. 
Therefore, if $u\in\anso{\Om}$ and $A\in\mA$, it follows that
$$F(u,A)\leq\liminf_h F_h(u,A)\leq G(u,A).$$
Finally, thanks again to Proposition \ref{quasisubprop} and repeating the proof of \cite[Theorem 18.7]{dalmaso}, we conclude that 
\begin{equation}\label{gammalimdue}
    F(\cdot,A)=\Gamma(W_X^{1,p})-\lim_{h\to+\infty}F_{h}(\cdot,A),
\end{equation}
for any $A\in\mA$.
\end{proof}
We have developed all the tools that we need to prove Theorem \ref{mainsobolev2}.

\begin{proof}[Proof of Theorem \ref{mainsobolev2}]
Since $(F_h)_h\subseteq \clag$, from Proposition \ref{firststepsobolev} we know that there exists a functional $F:\anso{\Om}\times\mA\scu[0,+\infty]$ which is a measure, local, satisfies \eqref{boundsobolev} and such that \eqref{gammalimsobolev} holds. Moreover, arguing as in the proof of Theorem \ref{709main1}, $F$ is convex. Therefore $F$ satisfies all the hypotheses of \cite[Theorem 2.3]{essverzpin}, and so we conclude that $F\in\clag$.
\end{proof}
In order to prove Theorem \ref{mainsobolev1}, we wish to apply \cite[Theorem 4.3]{essverzpin} to a suitable functional $F$. Anyway, among the other things, we need to guarantee that $F$ satisfies the strong condition $(\omega X)$.
With the two following propositions we are going to show that, whenever we work in $\clafto$, the strong condition $(\omega X)$ with respect to $\omega$ is preserved by the operation of $\Gamma(W_X^{1,p})$-limit.
\begin{prop}\label{strongborel}
If a functional $F:\anso{\Om}\times\mA\scu[0,+\infty]$ is a measure, it is $W_X^{1,p}$-continuous, it satisfies \eqref{boundsobolev} for any $u\in\anso{\Om}$ and for any $B\in\mathcal{B}$ and it satisfies the strong condition $(\omega X)$ with respect to $\omega$, then it holds that
\begin{equation}\label{newstronguno}
			|F(v,B')-F(u,B')|\leq\int_{B'}\omega_s(x,r) {\rm d} x
		\end{equation}
		for any $s\geq 0$, $B'\in\mB_0$, $r\geq 0$, $u,v\in \anso{\Om}$ such that 
		\begin{equation}\label{newstrongdue}
		\begin{aligned}
		|u(x)|,|v(x)|,|Xu(x)|,|Xv(x)|\leq s\\
		|u(x)-v(x)|,|Xu(x)-Xv(x)|\leq r
		\end{aligned}
		\end{equation}
		for a.e. $x\in B'$.
\end{prop}
\begin{proof}
It is not restrictive to assume that
$c_1=c_2=1$.
First we show the thesis for regular functions $u,v\in\anso{\Om}\cap C^\infty(\Om)$. Let us fix $B'\in\mB_0$, and $s,r,$ such that \eqref{newstrongdue} holds, and let us take $m,M>0$.
Since $F(u,\cdot)$ and $F(v,\cdot)$ are Borel measures, there exists a decreasing sequence of open sets $(A_n)_n\subseteq\mA$ such that $B'=\bigcap_{n=1}^\infty A_n$
and moreover
\begin{equation*}
    F(u,B')=\lim_{n\to\infty}F(u,A_n)\quad\text{and}\quad F(v,B')=\lim_{n\to\infty}F(v,A_n).
\end{equation*}
Furthermore, as $B'\Subset\Om$, we can assume that $A_n\Subset\Om$ for each $n\in\mathbb{N}$. Finally, as $u,v\in\ciuno{A_0}$ we can assume that 
\begin{align*}
		&|u(x)|,|v(x)|,|Xu(x)|,|Xv(x)|\leq s+\frac{1}{M}\\
		&|u(x)-v(x)|,|Xu(x)-Xv(x)|\leq r+\frac{1}{m}
\end{align*}
		for any $x\in A_n$ and for any $n\geq 0$.
We obtain that
\begin{equation*}
    \begin{split}
        |F(u,B')-F(v,B')|&=\lim_{n\to\infty}|F(u,A_n)-F(v,A_n)|\\
        &\leq\lim_{n\to\infty}\int_{A_n}w_{s+\frac{1}{M}}\left(x,r+\frac{1}{m}\right){\rm d} x\\
        &=\int_{B'}w_{s+\frac{1}{M}}\left(x,r+\frac{1}{m}\right){\rm d} x.
    \end{split}
\end{equation*}
Therefore, thanks to \eqref{incruno}, \eqref{incrdue} and the Monotone Convergence Theorem we conclude that 
\begin{equation*}
    \begin{split}
        \lvert F(u,B')-F(v,B') \rvert&\leq\lim_{m\to\infty}\lim_{M\to\infty}\int_{B'}w_{s+\frac{1}{M}}\left(x,r+\frac{1}{m}\right){\rm d} x\\
        &=\lim_{m\to\infty}\int_{B'}w_{s}\left(x,r+\frac{1}{m}\right){\rm d} x\\
        &=\int_{B'}w_{s}(x,r){\rm d} x.
    \end{split}
\end{equation*}
Let now $B'\in\mB_0$, $u,v\in\anso{\Om}$ and $s,r,$ such that \eqref{newstrongdue} holds, and fix again $m,M>0$. By Theorem \ref{2402meyersserrinx} there are two sequences $(u_h)_h,(v_h)_h\subseteq\anso{\Om}\cap C^\infty(\Om)$ converging respectively to $u$ and $v$ in the strong topology of $\anso{\Om}$. Therefore, thanks to the previous step and the continuity of the functional, we get that
\begin{equation*}
    \lvert F(u,B')-F(v,B') \rvert=\lim_{h\to\infty}|F(u_h,B')-F(v_h,B')|.
\end{equation*}
Now we want to estimate the right term.
For doing this let us define, for any $h\geq 0$,
\begin{align*}
A_h:=\left\{x\in B'\,:\,|u_h(x)|>s+\frac{1}{M}\right\}\quad B_h:=\left\{x\in B'\,:\,|v_h(x)|>s+\frac{1}{M}\right\}\\
C_h:=\left\{x\in B'\,:\,|Xu_h(x)|>s+\frac{1}{M}\right\}\quad D_h:=\left\{x\in B'\,:\,|Xv_h(x)|>s+\frac{1}{M}\right\}\\
E_h:=\left\{x\in B'\,:\,|u_h(x)-v_h(x)|>r+\frac{1}{m}\right\}\\ F_h:=\left\{x\in B'\,:\,|Xu_h(x)-Xv_h(x)|>r+\frac{1}{m}\right\},
\end{align*}
and let
\begin{equation}\label{zetah}
    Z_h:=A_h\cup B_h\cup C_h\cup D_h\cup E_h\cup F_h.
\end{equation}
We claim that 
\begin{equation*}
    \lim_{h\to\infty}|Z_h|=0.
\end{equation*}
Here we only show that $\lim_{h\to\infty}|A_h|=0$, being the other parts of the proof similar. Assume that $x\in A_h$ and assume that \eqref{newstrongdue} holds in $x$. Then it follows that
\begin{equation*}
    |u_h(x)-u(x)|\geq|u_h(x)|-|u(x)|> 
    \frac{1}{M}.
\end{equation*}
and hence 
$$x\in\left\{z\in\Om\,:\,|u(z)-u_h(z)|>\frac{1}{M}\right\}.$$
Since $u_h$ converges to $u$ in $\anso{\Om}$, then in particular $u_h$ converges to $u$ in measure, and so the measure of the right set goes to zero as $h$ goes to infinity. 
We can now estimate in this way.
\begin{equation*}
\begin{split}
    \lim_{h\to\infty}&|F(u_h,B')-F(v_h,B')|\leq\liminf_{h\to\infty}|F(u_h,B'\setminus Z_h)-F(v_h,B'\setminus Z_h)|+|F(u_h,Z_h)-F(v_h,Z_h)|\\
    &\leq\int_{B'}w_{s+\frac{1}{M}}\left(x,r+\frac{1}{m}\right)+\liminf_{h\to\infty}|F(u_h,Z_h)|+|F(v_h,Z_h)|\\
    &\leq\int_{B'}w_{s+\frac{1}{M}}\left(x,r+\frac{1}{m}\right){\rm d} x+\liminf_{h\to\infty}2\int_{Z_h}|a(x)|{\rm d} x\\
     &+\liminf_{h\to\infty}\int_{Z_h}|u_h|^p {\rm d} x+\int_{Z_h}|v_h|^p {\rm d} x+\int_{Z_h}|Xu_h|^p {\rm d} x+\int_{Z_h}|Xv_h|^p {\rm d} x\\
     &\leq\int_{B'}w_{s+\frac{1}{M}}\left(x,r+\frac{1}{m}\right) {\rm d} x+\liminf_{h\to\infty}2\int_{Z_h}|a(x)|{\rm d} x\\
     &+\liminf_{h\to\infty} 2^{p-1}\left(\int_{Z_h}|u_h-u|^p {\rm d} x+ \int_{Z_h}|u|^p {\rm d} x+ \int_{Z_h}|Xu_h-Xu|^p {\rm d} x+ \int_{Z_h}|Xu|^p {\rm d} x \right)\\
     &+\liminf_{h\to\infty} 2^{p-1}\left(\int_{Z_h}|v_h-v|^p {\rm d} x+  \int_{Z_h}|v|^p {\rm d} x+ \int_{Z_h}|Xv_h-Xv|^p {\rm d} x+ \int_{Z_h}|Xv|^p {\rm d} x \right)\\
     &\leq \int_{B'}w_{s+\frac{1}{M}}\left(x,r+\frac{1}{m}\right){\rm d} x +K\lim_{h\to\infty}\left(\|u-u_h\|_{\anso{\Om}}+\|v-v_h\|_{\anso{\Om}}\right)
    \\ &+\liminf_{h\to\infty}\int_{B'}\chi_{Z_h}b(x) {\rm d} x,
\end{split}
\end{equation*}
for a constant $K>0$ and a suitable function $b\in L^1(B')$. Therefore, thanks to the Dominated Convergence Theorem, we conclude that
\begin{equation*}
    |F(u,B')-F(v,B')|\leq\int_{B'}w_{s+\frac{1}{M}}\left(x,r+\frac{1}{m}\right){\rm d} x.
\end{equation*}
Arguing as in the first step and letting $M,m$ go to infinity, the thesis follows.
\end{proof}

\begin{prop}\label{strongcompatta}
Let $(F_h)_h$ be a sequence in $\clafto$ and assume that there exists a functional $F:\anso{\Om}\times\mA\scu[0,\infty]$ such that
\begin{equation*}\label{709gamma2}
    F(\cdot,A')=\Gamma(W_X^{1,p})-\lim_{h\to\infty}F_{h}(\cdot,A')\qquad\text{for any $A'\in\mA_0$.}
\end{equation*}
Then $F$ satisfies the strong condition $(\omega X)$ with respect to $\omega$.
\end{prop}
\begin{proof}
Let $A'\in\mA_0$, $u,v\in\anso{\Om}$ and $s,r\geq 0$ such that \eqref{newstrongdue} holds, and fix $m,M>0$.
Let $(u_h)_h$ and $(v_h)_h$ be recovery sequences respectively for $u$ and $v$.
Then it follows that
\begin{equation*}
    |F(u,A')-F(v,A')|=\lim_{h\to\infty}|F_h(u_h,A')-F_h(v_h,A')|.
\end{equation*}
Notice that, since $F_h\in\clafto$ then it is a measure, it satisfies the strong condition $(\omega X)$ with respect to $(\omega_s)_{s\geq 0}$, and thanks to a slight variant of \cite[Theorem 3.1]{essverzpin}, it is $W_X^{1,p}$-continuous. Moreover, thanks to \eqref{onlybounded}, it satisfies \eqref{boundsobolev} for any $u\in\anso{\Om}$ and for any $B\in\mathcal{B}$. Therefore it satisfies the hypotheses of Proposition \ref{strongborel}. Hence, repeating exactly the same estimates performed in the proof of Proposition \ref{strongborel}, we conclude that
\begin{equation*}
    \lim_{h\to\infty}|F_h(u_h,A')-F_h(v_h,A')|\leq\int_{A'}\omega_s(x,r){\rm d} x,
\end{equation*}
and so the thesis follows.
\end{proof}

We are now in position to give the proof of Theorem \ref{mainsobolev1}.

\begin{proof}[Proof of Theorem \ref{mainsobolev1}]
Since $(F_h)_h\subseteq\clafto$, from Proposition \ref{firststepsobolev} we know that there exists a functional $F:\anso{\Om}\times\mA\scu[0,\infty]$ which is a measure, local, $W_X^{1,p}$-lower semicontinuous and satisfies \eqref{boundsobolev}, and such that \eqref{gammalimsobolev} holds. Moreover, thanks to Proposition \ref{strongcompatta}, $F$ satisfies the strong condition $(\omega X)$ with respect to $\omega$. Therefore $F$ satisfies all the hypotheses of \cite[Theorem 4.4]{essverzpin}, and so we conclude that $F\in\clafto$.
\end{proof}

\section{Further Remarks and Open Problems}
The classical strong and weak condition $(\omega)$ were introduced in \cite{buttazzodalmaso2} in order to guarantee the continuity of the candidate Lagrangian when proving an integral representation result. In particular, the strong condition $(\omega)$ guarantees that $f(x,\cdot,\cdot)$ is continuous, while the weak condition $(\omega)$ implies the continuity of $f(x,\cdot,\xi)$. Moreover, it is easy to see that the strong condition $(\omega)$ implies the weak condition $(\omega)$. 
Anyway, in many situations it is difficult to verify the strong condition $(\omega)$, whereas the weak condition $(\omega)$ is easier. On the other hand, if we require only the weak condition $(\omega)$, we have to add an extra hypothesis in order to get the equivalence, i.e. the weak*-sequential lower semicontinuity of the functional, which is well known (cf. \cite{acerbifusco}) to be equivalent to the convexity of $f(x,u,\cdot)$. In \cite{essverzpin}, inspired by \cite{buttazzodalmaso2}, the authors exploited these ideas in order to achieve two integral representation results when the local functional is not assumed to be convex.
In Section 3 we obtained a $\Gamma(L^p)-$compactness result for a class of convex integral functionals defined on $\lp{\Om}$, but we did not generalized it when the convexity assumption is dropped. On the other hand, in Section 4 we considered also the non-convex case, working in a suitable class of integral functionals where the strong condition $(\omega X)$ is required uniformly on the class. Therefore there are some questions still unsolved.
Let us begin by properly extending Definition \ref{nuovastrongdef}.
\begin{deff}
    If $\omega=(\omega_s)_{s\geq 0}$ is a family of locally integrable moduli of continuity (cf. Definition \ref{nuovastrongdef}),
    we say that a functional $F:\lp{\Om}\times\mA\scu[0,+\infty]$ satisfies the \emph{weak condition $(\omega X)$ with respect to $\omega$} if
		\begin{equation}\label{2811strongprop}
			|F(u+r,A')-F(u,A')|\leq\int_{A'}\omega_s(x,|r|)\, {\rm d} x
		\end{equation}
		for any $s\geq 0$, $A'\in\mA_0$, $r\in\Ru$, $u\in \anso{\Om}$ such that 
		\begin{equation*}
		|u(x)|,|v(x)+r|,|r|\leq s
		\end{equation*}
		for a.e. $x\in A'$.
\end{deff}
Indeed, if $\omega$ is a fixed family of moduli of continuity it is reasonable to ask:
\begin{itemize}
    \item if the subclass of $\claft$ of those integral functionals satisfying the strong condition $(\omega X)$ with respect to $\omega$ is $\Gamma(L^p)-$compact;
    \item if the subclass of $\claft$ of those integral functionals satisfying the weak condition $(\omega X)$ with respect to $\omega$ and which are weakly*-seq. l.s.c. is $\Gamma(L^p)-$compact;
    \item if the subclass of $\clagt$ of those integral functionals satisfying the weak condition $(\omega X)$ with respect to $\omega$ and which are weakly*-seq. l.s.c. is $\Gamma(W_X^{1,p})-$compact.
\end{itemize}
In view of Proposition \ref{firststeplp}, Proposition \ref{firststepsobolev} and the integral representation results in \cite{essverzpin}, the only questions left open are the following.
\begin{itemize}
    \item[$(a)$] Is the $\Gamma(L^p)$-limit of a sequence of (possibly not weakly*-seq. l.s.c.) functionals a weakly*-seq. l.s.c functional?
    \item[$(b)$] Is the $\Gamma(W_X^{1,p})-$limit of a sequence of (possibly not weakly*-seq. l.s.c.) functional a weakly*-seq. l.s.c functional?
    \item[$(c)$] Does the $\Gamma(L^p)$-limit of a sequence of functionals satisfy the weak condition $(\omega X)$ provided that the sequence does satisfy it?
    \item[$(d)$] Does the $\Gamma(W_X^{1,p})-$limit of a sequence of functionals satisfy the weak condition $(\omega X)$ provided that the sequence does satisfy it?
    \item[$(e)$] Does the $\Gamma(L^p)$-limit of a sequence of functionals satisfy the strong condition $(\omega X)$ provided that the sequence does satisfy it?
\end{itemize}
Unfortunately we have not been able to answer to questions $(b),(c)$ and $(e)$. Anyway we are going to show that the questions $(a)$ and $(d)$ have a positive answer.

\begin{prop}[Answer to question $(d)$]\label{weakcompatta}
Let $\omega$ be a family of locally integrable moduli of continuity. Let $(F_h)_h$ be a sequence in $\clagt$ and assume that each $F_h$ satisfies the weak condition $(\omega X)$ with respect to $\omega$. Assume in addition that there exists a functional $F:\anso{\Om}\times\mA\scu[0,+\infty]$ such that
\begin{equation*}\label{709gamma2}
    F(\cdot,A')=\Gamma(W_X^{1,p})-\lim_{h\to+\infty}F_{h}(\cdot,A')\qquad\text{for any $A'\in\mA_0$.}
\end{equation*}
Then $F$ satisfies the weak condition $(\omega X)$ with respect to $\omega$.
\end{prop}
\begin{proof}
The proof of this result is totally similar to the proofs of Proposition \ref{strongborel} and Proposition \ref{strongcompatta}, and so we take it for granted.
\end{proof}

\begin{prop}[Answer to question $(a)$]
Let $F_h:\lp{\Om}\times\mA\scu[0,+\infty]$ be a sequence of (not necessary integral) functionals, and assume that there exists a functional $F:\lp{\Om}\times\mA\scu[0,\infty]$ which is a measure and such that
\begin{equation*}\label{709gamma2}
    F(\cdot,A')=\Gamma(L^p)-\lim_{h\to+\infty}F_{h}(\cdot,A')\qquad\text{for any $A'\in\mA_0$.}
\end{equation*}
Then $F$ is weakly*-seq. l.s.c
\end{prop}
\begin{proof}
Let $A\in\mA$, $A'\in\mA$ with $A'\Subset A$, $u\in\winf{\Om}$ and take a sequence $(u_h)_h\subseteq\winf{\Om}$ which is weakly*-convergent to $u$. Then, since $A'\Subset A$, it is well known that $u_h$ converges to $u$ strongly in $L^\infty(A')$, and so in particular strongly in $L^p(A')$. Being $F(\cdot,A')$ a $\Gamma(L^p)-$limit, it is $L^p-$lower semicontinuous. Moreover, being $F$ a measure, it is also increasing. These facts imply that
\begin{equation*}
    F(u,A')\leq\liminf_{h\to\infty}F(u_h,A')\leq\liminf_{h\to\infty}F(u_h,A).
\end{equation*}
Since $F$ is inner regular and since $A'\Subset A$ is arbitrary, the conclusion follows.
\end{proof}

\end{document}